\documentclass[12pt]{amsart}

\usepackage{amsmath, amssymb, amscd, verbatim, xspace,amsthm}
\usepackage{latexsym, epsfig, color}
\usepackage[hidelinks]{hyperref}

\usepackage[OT2,T1]{fontenc}
\DeclareSymbolFont{cyrletters}{OT2}{wncyr}{m}{n}
\DeclareMathSymbol{\Sha}{\mathalpha}{cyrletters}{"58}

\newcommand{\ba}{\begin{align*}}
\newcommand{\ea}{\end{align*}}

\newcommand{\A}{\ensuremath{{\mathbb{A}}}}
\newcommand{\C}{\ensuremath{{\mathbb{C}}}}
\newcommand{\Z}{\ensuremath{{\mathbb{Z}}}\xspace}
\renewcommand{\P}{\ensuremath{{\mathbb{P}}}}

\newcommand{\Q}{\ensuremath{{\mathbb{Q}}}}
\newcommand{\R}{\ensuremath{{\mathbb{R}}}}
\newcommand{\F}{\ensuremath{{\mathbb{F}}}}

\newcommand{\ra}{\rightarrow}

\newcommand\Conf{\operatorname{Conf}}

\newcommand\Aut{\operatorname{Aut}}

\newcommand\im{\operatorname{im}}

\newcommand\Gal{\operatorname{Gal}}
\newcommand\Nm{\operatorname{Nm}}

\newcommand\Prob{\operatorname{Prob}}
\newcommand\Sur{\operatorname{Sur}}

\newcommand\Tr{\operatorname{Tr}}

\newcommand\tensor{\otimes}
\newcommand\isom{\simeq}
\newcommand\sub{\subset}
\newcommand\tesnor{\otimes}

\newcommand\Disc{\operatorname{Disc}}

\newcommand\PGL{\operatorname{PGL}}

\newcommand\Spec{\operatorname{Spec}}

\newcommand\Jac{\operatorname{Jac}}

\newcommand\Frob{\operatorname{Frob}}

\renewcommand\O{\mathcal{O}}

\newcommand\Pic{\operatorname{Pic}}

\newcommand\bq{\begin{equation}}
\newcommand\eq{\end{equation}}

\newtheorem{proposition}{Proposition}[section]
\newtheorem{theorem}[proposition]{Theorem}
\newtheorem{corollary}[proposition]{Corollary}

\newtheorem{lemma}[proposition]{Lemma}

\newtheorem{conjecture}[proposition]{Conjecture}

\theoremstyle{remark}
\newtheorem{remark}[proposition]{Remark}

\usepackage{fullpage, pdfsync}

\newcommand\Cl{\operatorname{Cl}}

\renewcommand\A{\mathcal{A}}
\newcommand\Ad{\mathcal{A}_\bullet}
\newcommand{\CCHur}{\mathsf{CHur}}
\newcommand{\CHur}{\operatorname{CHur}}
\renewcommand{\v}{\infty}
\newcommand\GG{H}
\newcommand{\SC}{C}

\title{Cohen-Lenstra heuristics and local conditions}

\author{Melanie Matchett Wood}
\address{Department of Mathematics\\
University of Wisconsin-Madison \\ 480 Lincoln Drive \\
Madison, WI 53705 USA\\
and
American Institute of Mathematics\\600 East Brokaw Road\\
San Jose, CA 95112 USA}  
\email{mmwood@math.wisc.edu}

\begin{document}

\begin{abstract}
We prove function field theorems supporting the Cohen-Lenstra heuristics for real quadratic fields, and natural strengthenings of these analogs from the affine class group to the Picard group of the associated curve.  
Our function field theorems also support a conjecture of Bhargava on how local conditions on the quadratic field do not affect the distribution of class groups.
Our results lead us to make further conjectures refining the Cohen-Lenstra heuristics, including on the distribution of certain elements in class groups.  We prove instances of these conjectures in the number field case.  Our function field theorems use a homological stability result of Ellenberg, Venkatesh, and Westerland.
\end{abstract}

\maketitle

\section{Introduction}

For any odd prime $p$, Cohen and Lenstra \cite{CL84} conjectured the distribution of the Sylow $p$-subgroups of class groups of imaginary and real quadratic fields. 
 In particular, if we consider the measure on finite abelian $p$-groups such that 
$$\mu_{CL}(A):=\frac{1}{|\Aut(A)|}\prod_{i\geq 1}(1-p^{-i}),$$
they conjectured that this measure gives the distribution of the Sylow $p$-subgroups of class groups of imaginary quadratic fields.
For real quadratic fields, we make a probability measure $\mu^r_{CL}$ on finite abelian $p$-groups by producing a random group as follows: pick a random group $B$ with respect to $\mu_{CL}$, then pick a (uniform) random element $b\in B$, and then form $B/\langle b \rangle$.  Then $\mu^r_{CL}(A)$ is the probability that this process produces a group isomorphic to $A$.  They then conjectured that $\mu^r_{CL}$ gives the distribution of the Sylow $p$-subgroups of class groups of real quadratic fields.

One of the most compelling reasons to believe this modification for the real quadratic case is by considering the function field analog.
For a finite field $\F_q$ and an extension $K/\F_q(t)$, let $\O_K$ be the integral closure of $\F_q[t]$ in $K$ and $C_K$ be the smooth projective curve over $\F_q$ associated to $K$.
 When $K/\F_q(t)$ is imaginary quadratic (i.e. ramified over $\infty$), then $\Cl(\O_K)\isom\Pic^0(C_K)$.
However, when $K/\F_q(t)$ is real quadratic (i.e. $\infty$ splits into $\infty_1,\infty_2$), then
$\Cl(\O_K)\isom\Pic^0(C_K)/\langle\infty_1-\infty_2 \rangle$. 
The most direct analogy of the Cohen-Lenstra heuristics from the number field case to the function field case asks about the distribution of $\Cl(\O_K)$.  However, in the real quadratic function field case,
 we can ask the richer question of the distribution  of the 
Sylow $p$-subgroups of $\Pic^0(C_K)$ and of $\infty_1-\infty_2\in \Pic^0(C_K).$
In this paper, we prove a theorem about the actual distribution for this richer question.  Before stating these results, we will explain a natural heuristic that will predict what we obtain.

A \emph{pointed abelian $p$-group} is a pair $(A,a)$ where $A$ is an abelian $p$-group and $a\in A$, and two 
pointed abelian $p$-groups $(A,a)$ and $(B,b)$ are isomorphic if there is an isomorphism $A\ra B$ taking $a$ to $b$. 
We put a measure $\mu$ on the set of isomorphism classes of pointed abelian $p$-groups as follows.  
Pick a random group $B$ with respect to $\mu_{CL}$, then pick a (uniform) random element $b\in B$.  Then $\mu(A,a)$
is the probability that this process produces a pointed group isomorphic to $(A,a)$.
Our heuristic is that $\mu$ gives the distribution of $(\Pic^0(C_K),\infty_1-\infty_2)$.

In the progress that has been made on proving exact Cohen-Lenstra predictions for class groups (e.g. \cite{DH71, DW88, Bha05, FK07, Ellenberg2016}), the Cohen-Lenstra measures have often been accessed through their moments---moments that determine a unique distribution.
So, we now turn to the moments of $\mu$ (see \cite[Section 3.3]{Clancy2015} for an explanation of the terminology ``moments'' in this situation). 
We let $\Sur((B,b),(A,a))$ denote the surjective homomorphisms from $B$ to $A$ that take $b$ to $a$. Let $\Ad$ be a set with one representative from each isomorphism class of pointed finite abelian $p$-groups.
\begin{lemma}[Pointed moments]\label{L:mom}
For any finite abelian $p$-group $A$ and $a\in A$, we have
$$
\sum_{(B,b)\in\Ad} |\Sur((B,b),(A,a))|\mu(B,b)=\frac{1}{|A|}.
$$
\end{lemma}
We call this average $\sum_{(B,b)\in\Ad} |\Sur((B,b),(A,a))|\mu(B,b)$ the \emph{ $(A,a)$ moment} of $\mu$.
In fact, these moments determine the measure $\mu$.
\begin{lemma}[Pointed moments determine distribution]\label{L:pmomdet}
 If $\nu$ is a measure on isomorphism classes finite abelian, pointed $p$-groups such that for all $(A,a)\in\Ad$ we have
$$
\sum_{(B,b)\in\Ad} |\Sur((B,b),(A,a))|\nu(B,b)=\frac{1}{|A|},
$$
then $\nu=\mu$.
\end{lemma}

We prove the following, which gives evidence towards the Cohen-Lenstra
heuristics over function fields as well as the refined heuristic above.
\begin{theorem}[Pointed $\Pic^0(C_K)$ distribution]\label{T:RSC}
Let $A$ be a finite odd order abelian group, let $a\in A$, and let
 $$
\delta^+_q:=\limsup_{m \ra\infty}  \frac{  \sum_{K/\F_q(t) \textrm{ real quad,} \Nm \Disc(K)=q^{2m}} |\Sur((\Pic^0(C_K),\infty_1-\infty_2),(A,a))|}{\sum_{K/\F_q(t) \textrm{ real quad,} \Nm \Disc(K)=q^{2m}}1}
$$
and $\delta^-_q$ the corresponding $\liminf$.
Then as $q\ra\infty$ among odd prime powers such that $(q(q-1),|A|)=1$, we have
$$
\delta^+_q,\delta^-_q\ra \frac{1}{|A|}.
$$
\end{theorem}
The function field analog of our refined heuristic above predicts that for each odd prime power $q$ with $(q(q-1),|A|)=1$ that $\delta_q^\pm=|A|^{-1}$.
In the case $a=0$ we have the usual group-indexed moments of $\Cl(\O_K)$ (Corollary~\ref{C:regmom})
and adding over $a$ we have the usual group-indexed moments of $\Pic^0(C_K)$ (Corollary~\ref{C:Picmom}). 

We prove Theorem~\ref{T:RSC} by first converting the problem of counting pointed surjections to a problem of counting extensions of $\F_q(t)$, which are parametrized by a Hurwitz scheme first defined Romagny and Wewers \cite{Romagny2006}, and further studied by Ellenberg, Venkatesh, and Westerland \cite{Ellenberg2012}.
We use the Grothendieck-Lefschetz trace formula to count $\F_q$ points on these schemes.  This overall approach to the Cohen-Lenstra heuristics over function fields goes back to to unpublished work of J.-K. Yu, and was also built upon by Achter \cite{Achter2006}.

Ellenberg, Venkatesh, and Westerland made a tremendous breakthrough in this approach when they proved a homological stability theorem for the complex versions of these Hurwitz schemes \cite{Ellenberg2016}.
Also, Ellenberg, Venkatesh, and Westerland's work on Hurwitz schemes  relates their components to group theoretic invariants \cite{Ellenberg2012}. Both of these advances will go into the proof of Theorem~\ref{T:RSC}.
Ellenberg, Venkatesh, and Westerland \cite{Ellenberg2016} have proven the non-pointed analog of Theorem~\ref{T:RSC} (i.e., the analog of Corollary~\ref{C:Picmom}) for imaginary quadratic extensions, which was the motivation for their work on Hurwitz schemes.

In \cite{CL84}, Cohen and Lenstra actually give a conjecture for the entire odd part of the class groups of quadratic fields. 
In this paper, for simplicity we  restrict ourselves to Sylow $p$-subgroups when explaining the heuristics, but give our theorems about class groups without this restriction.
Note that all results support the heuristic that $p_i$-parts of class group behavior are independent at finite sets of primes $p_i$.

We prove a similar result (Theorem~\ref{T:in}) for quadratic extensions of $\F_q(t)$ inert over $\infty$ for non-pointed moments.
 Combining these results with those of Ellenberg, Venkatesh, and Westerland, we come to the following conclusion.
  \emph{Among all quadratic $K/\F_q(t)$ we have proven evidence that the Sylow $p$-subgroups of $\Pic^0(C_K)$
are distributed by $\mu_{CL}$, and further that this distribution is unaffected by restricting to quadratic extensions with a certain behavior at $\infty$.}  
As long as we use $\Pic^0(C_K)$, as opposed to $\Cl(\O_K)$, we see there is no difference between the real and imaginary distributions.
Indeed, this gives evidence for the following conjecture, which was made by Bhargava for number fields at the 2011 AIM Workshop on the Cohen-Lenstra heuristics and the subject of many days of discussion at the workshop.
(For a number field $K$, let $\O_K$ denote its ring of integers. For a finite abelian group $A$, let $A_{odd}$ denote its odd part.)
\begin{conjecture}[Local conditions on $K$, c.f. \cite{Bha11} ]\label{C:indoffin}
In any of the following families, ordered by discriminant, with their class groups as given:
\begin{enumerate}
 \item $K/\Q$ imaginary quadratic, $\Cl(\O_K)_{odd}$
 \item $K/\Q$ real quadratic, $\Cl(\O_K)_{odd}$
\item $K/\F_q(t)$ quadratic, $\Pic^0(C_K)_{odd}$
\end{enumerate}
the distribution of the respective class groups is not changed upon restricting only  those $K$ that have specific completions at a finite set of nonarchimedian places (of $\Q$ or $\F_q(t)$).  In particular, in $(3)$ the Sylow $p$-subgroups are distributed by $\mu_{CL}$
when $p\nmid q(q-1)$.
\end{conjecture}

By the action of $\PGL_2(\F_q)$ in Theorem~\ref{T:RSC}, we could replace $\infty$ by any degree $1$ point of $C_K$, and so we have evidence for Conjecture~\ref{C:indoffin}(3), for a condition at any degree $1$ place.

The prediction of Conjecture~\ref{C:indoffin} for the average size of $3$-torsion in the class groups of real or imaginary quadratic fields was proven by Bhargava and Varma  \cite[Corollary 4]{Bhargava2016}.
Evidence for a generalization of Conjecture~\ref{C:indoffin} to cubic extensions is given by Bhargava and Varma in \cite{Bhargava2015d},
in which they prove the average size of $2$-torsion of class groups of cubic fields is not affected by local conditions at non-archimedian places.

In Conjecture~\ref{C:indoffin}, we could lump together the two cases over $\Q$, so that the distribution would be the average of $\mu_{CL}$ and $\mu^r_{CL}$ and make the same statement.  The work of Davenport and Heilbronn \cite{DH71} shows that the distribution is indeed (provably) affected by restricting to only those $K/\Q$ with a specific completion at $\infty$.  It is worth noting that the distribution of class groups of cubic fields is also  provably affected by certain global conditions on the field, such as monogenicity, by work of Bhargava and Shankar \cite{BS10v2}. 

Further, Theorem~\ref{T:RSC} leads us to conjecture the following on the distribution of special elements in  class groups.  For an abelian group $A$, let $A_p$ denote the Sylow $p$-subgroup of $A$.  
\begin{conjecture}[Distribution of elements in $\Cl_K$]\label{C:atsc}
Let $p$ be an odd rational prime and $Q$ be $\Q$ or $\F_q(t)$ with $p\nmid q(q-1)$. For $K/Q$ let $\Cl_K$ be $\Cl(\O_K)_{p}$ or $\Pic^0(C_K)_{p}$ respectively, and use $+$ to denote the group law.
For a finite place $v$ of $Q$, if we consider quadratic extensions $K$ of $Q$ that are split completely at $v$ into $v_1,v_2$, ordered by discriminant, then (informally) $v_1-v_2$ is distributed uniformly in $\Cl_K$.  More precisely,
in the number field case if we restrict to $K$ imaginary, or in the function field case, the isomorphism classes of pairs $(\Cl_K,v_1-v_2)$ are distributed according to $\mu$.
\end{conjecture}

Theorem~\ref{T:RSC} proves evidence towards Conjecture~\ref{C:atsc} in function field cases.  We also prove the  prediction of Conjecture~\ref{C:atsc} on the average $3$-torsion of the class groups in the number field case, building on the work of Davenport and Heilbronn \cite{DH71} 
 (which proves the original Cohen-Lenstra prediction for the average size of the $3$-torsion of the class group). For example, our result in the imaginary quadratic number field case is as follows (see Theorem~\ref{T:DH2} for the complete result).

\begin{theorem}[Distribution of elements in $\Cl_K$, $\Z/3\Z$ moment]\label{T:iDH2}
 Let $v_1,\dots,v_n$ be finite places of $\Q$, and $F_i$ be \'{e}tale quadratic  $\Q_{v_i}$-algebras, 
with $F_1=\Q_{v_1}^{\oplus 2}$.  
Let $S^+$ be the set of imaginary quadratic extensions $K$ of $\Q$ such that $K\tensor_\Q \Q_{v_i}\isom  F_i$.
Let $v_1$ split into $w_1$ and $w_2$ in $K$ and let $a\in \Z/3\Z$.
We have
$$
\lim_{X\ra\infty} \frac{\sum_{K\in S^+, |\Disc(K)|<X} |\Sur((\Cl(\O_K),w_1-w_2),(\Z/3\Z,a)| }{\sum_{K\in S^+, |\Disc(K)|<X} 1}
=\frac{1}{3}.
$$
\end{theorem}

Note that Theorem~\ref{T:iDH2} also provides evidence for the philosophy of Conjecture~\ref{C:indoffin} as the result is not changed by finitely many local conditions at finite places.
  Klagsbrun \cite{Klagsbrun2017a} determines the $\Z/3\Z$-moments of quotients of the class groups of random quadratic fields by primes above a fixed set of primes, which provides evidence for a generalization of Conjecture~\ref{C:atsc} to multiple places in the base field.
 Evidence for a generalization of Conjecture~\ref{C:atsc} to cubic extensions, as well as to multiple places in the base field, is given by Klagsbrun in \cite{Klagsbrun2017}, in which he determines the $\Z/2\Z$-moments of quotients of the class groups of cubic fields by primes above a fixed set of primes.

Our conjectures lead to some interesting predictions discussed in Section~\ref{S:conjs}.  We predict that the number of $\F_q$-points on a hyperelliptic curve does not affect the number of $\F_q$-points on the $p^k$-torsion of its Jacobian (even though all of the $\F_q$-points of the Jacobian are torsion points).
We predict that among imaginary quadratic extensions $K/\Q$ split completely at a rational prime $\ell$
into $\ell_1,\ell_2$, the distribution of $\Cl(\O_K)_p/\langle \ell_1-\ell_2 \rangle$ is not changed if we restrict to only those $K$ for which
$\ell_1-\ell_2$ is trivial in $\Cl(\O_K)_p.$

\subsection{Outline of the paper}
In Section~\ref{S:meas}, we describe the Cohen-Lenstra measures, their moments, and show that the moments determine the measures.  In Section~\ref{S:pmeas}, we do the same for the analog for pointed groups, and in particular we prove Lemmas~\ref{L:mom} and \ref{L:pmomdet}.  In Section~\ref{S:rproof}, we prove Theorem~\ref{T:RSC}, giving evidence towards the pointed moments of real quadratic Picard groups over rational function fields.  As corollaries, we obtain the non-pointed moments.  In Section~\ref{S:iproof}, we prove pointed moments of inert (at $\infty$) quadratic Picard groups over rational function fields.  
In Section~\ref{S:CLd}, we prove instances of Conjecture~\ref{C:atsc} for number fields, including Theorem~\ref{T:iDH2} (in Theorem~\ref{T:DH2}).  In Section~\ref{S:conjs}, we discuss some interesting consequences of our conjectures.

\subsection{Notation}
Throughout the paper we fix an odd prime $p$.
The letter $q$ will always denote a prime power.
 Let $\Sur(A,B)$ denote the surjective homomorphisms from $A$ to $B$.  For elements $g_i$ of an abelian group, we let $\langle g_1,\dots,g_u\rangle$ denote the subgroup generated by the $g_i$.  Let $\A$ be the set of finite abelian $p$-groups (or more precisely, a set with one representative from each isomorphism class of finite abelian $p$-groups).  For a group $A$, we write $A_p$ for the Sylow $p$-subgroup of $A$. 

\section{Cohen-Lenstra measures and moments}\label{S:meas}
In this section we will explain the Cohen-Lenstra measures on $\A$ and the moments of these measures.  
Each measure will be with respect to the $\sigma$-algebra that is the full power set of $\A$.
Let $u$ be an integer, $u\geq 0$.
  We define the Cohen-Lenstra measure $\mu^u$, following \cite[Example 5.9]{CL84}, so that for $A\in \A$, we have
$$
 \mu^{u}(A):=\frac{1}{|A|^{u}|\Aut(A)|}\prod_{k=1}^{\infty}(1-p^{-k-u}).$$
In particular, for $u=0$ the measure (called $\mu_{CL}$ above) describes the predicted distribution of Sylow $p$-subgroups of imaginary quadratic class groups and for $u=1$ the measure (called $\mu_{CL}^r$ above) describes the predicted distribution of Sylow $p$-subgroups of real quadratic class groups (see \cite[8.1]{CL84}).

We now give the moments of these distributions.  

\begin{proposition}[Lemma 3.2 of \cite{Wood2015a}]\label{P:mom}
For any $A\in \A$, we have that
$$
\sum_{B\in \A} |\Sur(B,A)| \mu^u(B)=|A|^{-u}.
$$
\end{proposition}

In particular, applying Proposition~\ref{P:mom} when $A=1$ shows that $\mu^u$ is a probability measure.
We now give another description of $\mu^u$ (see also \cite[Theorem 10.9]{Lengler} for this description of $\mu^u$).

\begin{proposition}
Take a random group $G$ according to $\mu^0$, and then select $u$ elements $g_i\in G$ uniformly and independently at random.  Then $G/\langle g_1,\dots,g_u \rangle$ is distributed according to $\mu^u$.
\end{proposition}

\begin{proof}
Let $K$ be a finite abelian $p$-group of order $p^b$.
If we apply \cite[Proposition 4.3]{CL84}  (with their $k=\infty$),
we find that for any integers $a\geq b$,
\begin{equation}\label{E:fixedsize}
\sum _{\substack{B\in \A\\|B|=p^a}}\mu^0(B) \frac{1}{|B|^u} \sum_{\substack{b_1,\dots, b_u\in B\\ B/\langle b_1,\dots, b_u \rangle \isom K}}1=p^{-bu}   \frac{1}{|\Aut(K)|} \prod_{k=1}^{\infty}(1-p^{-k})  \sum_{\substack{C\in \A\\|C|=p^{a-b}}} \frac{|\Sur(\Z^u,C)|}{|C|^u|\Aut(C)|}.
\end{equation}
By \cite[Corollary 3.7]{CL84} (with their $k=u$ and $s=0$), we have 
$$
\sum_{a\geq b} \sum_{\substack{C\in \A\\|C|=p^{a-b}}} \frac{|\Sur(\Z^u,C)|}{|C|^u|\Aut(C)|}=\prod_{1\leq j\leq u} (1-p^{-j})^{-1}.
$$
Thus, summing Equation~\eqref{E:fixedsize} over all $a\geq b$, we have
$$
\sum _{B\in \A}\mu^0(B) \frac{1}{|B|^u} \sum_{\substack{b_1,\dots, b_u\in B\\ B/\langle b_1,\dots, b_u \rangle \isom K}}1= \frac{1}{|K|^u|\Aut(K)|} \prod_{k=1}^{\infty}(1-p^{-k}) \prod_{1\leq j\leq u} (1-p^{-j})^{-1}.
$$
The left-hand side is the probability of producing $K$ via the process described in the proposition, and the right-hand side is $\mu^u(K).$
\end{proof}

In \cite{CL84}, it is this second description (of $\mu^1$) that in fact comes first, and the probabilities for individual groups are later computed in \cite[Example 5.9]{CL84}. 
One can also prove that the value $\mu^u(B)$ is the limit as $n\ra \infty$ of the probability that $B$ is the cokernel of a random matrix from Haar measure in $M_{n \times (n+u)}(\Z_p)$, as done by Friedman and Washington \cite{FW89} in the case $u=0$ (using the same argument as Friedman and Washington).  At first, it seems this could give a convenient proof of the moments in Proposition~\ref{P:mom}.  However, since the limit in $n$ does not a priori commute with the sum over $B$ (though in fact this can be shown), this approach turns out to be  less convenient than that above.

In fact, the moments in Proposition~\ref{P:mom} determine the measure $\mu^u.$

\begin{proposition}\label{P:momdet}
 If $\mu$ is a measure on $\A$ such that for every $A\in \A$ we have
$$\sum_{B\in \A} \Sur(B,A) \nu(B)=|A|^{-u},$$
then
$\nu=\mu^u.$
\end{proposition}

The $u=0$ case of Proposition~\ref{P:momdet} is Lemma 8.2 in \cite{Ellenberg2016}. 
Proposition~\ref{P:momdet}  follows from the proof of \cite[Theorem 8.3]{Wood2017} (see \cite[Theorem 3.1]{Wood2015a} for a statement).
 We give a much simpler proof here following \cite{Ellenberg2016} and in particular using the following lemma of infinite dimensional linear algebra, the argument for which is given in 
the proof of Lemma 8.2 in \cite{Ellenberg2016} (see also \cite[Lemma 4.7]{Boston2017} for a stronger statement).

\begin{lemma}\label{L:la}
Let $a_{i,j}$ be non-negative real numbers indexed by pairs of natural numbers $i,j$, such that
there is an $\alpha<2$ so that 
 for all
$i$, we have $a_{i,i}=1$ and $\sum_{j} a_{ij}=\alpha$.  Let $x_j,y_j$ be non-negative reals
indexed by natural numbers $j$.  If for all $i,$
$$
\sum_j a_{i,j} x_j=\sum_j a_{i,j}y_j =1,
$$ 
then $x_j=y_j$ for all $j$.
\end{lemma}

\begin{proof}[Proof of Proposition~\ref{P:momdet}]
Enumerate the finite abelian $p$-groups as $A_i.$
 We apply Lemma~\ref{L:la} with 
$$a_{i,j}=\frac{|A_i|^u|\Sur(A_j,A_i)|}{|A_j|^u|\Aut(A_j)|}$$
and $x_j=|A_j|^u|\Aut(A_j)|\nu(A_j)$ and $y_j=|A_j|^u|\Aut(A_j)|\mu^u(A_j)$.
It remains to check that $$\sum _j \frac{|A_i|^u|\Sur(A_j,A_i)|}{|A_j|^u|\Aut(A_j)|} <2.$$
However, we have that 
 $$
\sum _j \frac{|A_i|^u|\Sur(A_j,A_i)|}{|A_j|^u|\Aut(A_j)|} =
\sum _j |A_i|^u|\Sur(A_j,A_i)| \mu^{u}(A_j)\prod_{k=1}^{\infty}(1-p^{-k-u})^{-1}=\prod_{k=1}^{\infty}(1-p^{-k-u})^{-1}.
$$
So it remains to check that $\prod_{k=1}^{\infty}(1-p^{-k-u})>1/2.$  The expression is decreasing 
in $p$ and $u$, so it remains to check for $p=3$ and $u=0$, which can be done simply.
\end{proof}

\section{Measures and moments for pointed groups}\label{S:pmeas}

We have the measure $\mu$ on $\Ad$ defined in the introduction by choosing a group according to $\mu^0$ and picking a group element uniformly at random.  We will now analyze the moments of this measure analogously to the Cohen-Lenstra measures $\mu^u$ in Section~\ref{S:meas}.

Recall from the introduction, we have the following.

\noindent
{\bf Lemma~\ref{L:mom}.} (Pointed moments)
For any finite abelian $p$-group $A$ and $a\in A$, we have
$$
\sum_{(B,b)\in\Ad} |\Sur((B,b),(A,a))|\mu(B,b)=\frac{1}{|A|}.
$$

\begin{proof}
We have 
\begin{align*}
\sum_{(B,b)\in\Ad} |\Sur((B,b),(A,a))|\mu(B,b)&=\sum_{B\in\A} \frac{1}{|B|} \sum_{b\in B}|\Sur((B,b),(A,a))|\mu^0(B)\\
&=\sum_{B\in\A} \frac{\mu^0(B)}{|B|} \sum_{\phi\in \Sur(B,A)} |\ker(\phi)|\\
&=\sum_{B\in\A} \frac{\mu^0(B)}{|A|} |\Sur(B,A)|=\frac{1}{|A|},\\
\end{align*}
where the last equality is by Proposition~\ref{P:mom}.
\end{proof}

\begin{lemma}
We have $$\mu(B,b)=\frac{1}{|B||\Aut(B,b)|}\prod_{k=1}^\infty (1-p^{-k}).$$
\end{lemma}
\begin{proof}
As $\Aut(B)$ acts on the elements of $B$, we have
$|\Aut(B)|=|\Aut(B,b)|\cdot\#\{h\in B\ |\ (B,h)\isom (B,b)\},$ and the lemma follows.
\end{proof}

In fact, the moments of Lemma~\ref{L:mom} determine the distribution $\mu$.

\noindent
{\bf Lemma~\ref{L:pmomdet}.}
 If $\nu$ is a measure on $\Ad$ such that for every $(A,a)\in\Ad$ we have
$$
\sum_{(B,b)\in\Ad} |\Sur((B,b),(A,a))|\nu(B,b)=\frac{1}{|A|},
$$
then $\nu=\mu$.

\begin{proof}
 Enumerate the finite abelian pointed $p$-groups as $(A_i,a_i).$
 We apply Lemma~\ref{L:la} with 
$$a_{i,j}=\frac{|A_i||\Sur((A_j,a_j),(A_i,a_i))|}{|A_j||\Aut(A_j,a_j)|}$$
and $x_j=|A_j||\Aut(A_j,a_j)|\nu(A_j,a_j)$ and $y_j=|A_j||\Aut(A_j,a_j)|\mu(A_j,a_j)$.
It remains to check that $$\sum _j \frac{|A_i||\Sur((A_j,a_j),(A_i,a_i))|}{|A_j||\Aut(A_j,a_j)|}<2.$$
However, we have that 
\begin{align*}
 \sum _j \frac{|A_i||\Sur((A_j,a_j),(A_i,a_i))|}{|A_j||\Aut(A_j,a_j)|}& =
\sum _j |A_i||\Sur((A_j,a_j),(A_i,a_i))| \mu(A_j,a_j)\prod_{k=1}^{\infty}(1-p^{-k})^{-1}\\&=\prod_{k=1}^{\infty}(1-p^{-k})^{-1}.
\end{align*}
So it remains to check that $\prod_{k=1}^{\infty}(1-p^{-k})>1/2,$
as in Proposition~\ref{P:momdet}.
\end{proof}

\section{Theorems for real quadratic function fields}\label{S:rproof}
In this section we prove Theorem~\ref{T:RSC}.  The overall strategy is the same as in \cite[Sections 5 and 6]{Boston2017} and \cite[Section 3]{Wood2017a}, which both prove function field results about non-abelian analogs of class groups.
However, many of the details are different and the argument below is self-contained.  We
first translate the problem of interest into one of counting certain extensions of $\F_q(t)$.  We then will use the existence of a Hurwitz scheme parametrizing such extensions (as they are equivalently curves with a map to the line), which comes from work of 
Ellenberg, Venkatesh, and Westerland \cite{Ellenberg2012}, building on work of Romagny and Wewers \cite{Romagny2006}.
Unlike in \cite{Boston2017} and \cite{Wood2017a}, in this paper we also use the homological stability results of 
Ellenberg, Venkatesh, and Westerland \cite{Ellenberg2016} to have a bound on the $i$th cohomology groups of the Hurwitz schemes that is exponential in $i$.

\subsection{Notation}\label{S:FFnot}
Let $Q=\F_q(t)$ for this section and the next.  
Let $SCQ$ be the set of all quadratic extensions of $Q$ split completely at $\infty$.
For $K\in SCQ$, let $\infty_1, \infty_2$ be the two places
of $K$ over $\infty$.  
We define $\Pic(C_K)$ to be the Picard group of the unique smooth, proper curve $C_K$ over $\F_q$ associated to $K$.   Let $M_K$ be the set of places of $K$, which are in bijection with the closed points of $C_K$. 
We have
$$
\Pic(C_K)=\left(\bigoplus_{v\in M_K} \Z v\right)/\{div(f)\ |f\in K \}.
$$
There is a natural map $\Pic(C_K)\ra \Z$ sending a place $v$ corresponding to a closed point of degree $d$ to $d$, with kernel $\Pic^0(C_K)$.  Let $K^{un,ab}$ be the maximal unramified abelian extension of $K$.  By class field theory, we have that the profinite completion $\widehat{\Pic(C_K)}$ is isomorphic to $\Gal(K^{un,ab}/K)$.

Note that $\delta_K:=\infty_1-\infty_2$ is not a well defined element in $\Pic^0(C_K)$ because it depends on the ordering of $\infty_1,\infty_2$, but it is well defined up to $\pm$, and thus the isomorphism class of the pointed group $(\Pic^0(C_K),\delta_K)$ is well-defined.  The remainder of this section is devoted to the proof of the pointed moments of Picard groups of real quadratic function fields.

\subsection{Counting pointed surjections}
In order to prove Theorem~\ref{T:RSC}, we will first translate the problem from one of counting pointed surjections to one of counting certain extensions of $K$.
If $K/\F_q(t)$ is a quadratic extension corresponding to $\phi: C_K \ra \P^1_{\F_q}$, we have $\phi^*(\infty)\in \Pic(C_K)$, where if $\infty$ splits into $\prod_i \infty_i^{e_i}$, we have $\phi^*(\infty)=\sum_i e_i\infty_i.$

\begin{proposition}\label{P:whichL}
Let $K/\F_q(t)$ be a quadratic extension corresponding to the map $\phi: C_K \ra \P^1_{\F_q}$ of curves over $\F_q$, and let $d$ be the greatest common divisor of the degrees of the points of $C_K$.
Let $L$ be a subfield $K\sub L \sub K^{un,ab}$ and let $P\sub \widehat{\Pic(C_K)}$ be the corresponding
subgroup via Galois theory.  Then $d\phi^*(\infty) \in P$ if and only if 
\begin{enumerate}
\item $L/Q$ is Galois, and
\item $\Gal(K/Q)$, by conjugation in $\Gal(L/Q)$, acts as inversion on $\Gal(L/K).$ 
\end{enumerate}
Also, since $d|2$, if $P$ is odd index, then $d\phi^*(\infty) \in P$ if and only if $\phi^*(\infty) \in P.$
\end{proposition}

\begin{proof}
Let $\sigma$ be the generator of $\Gal(K/Q)$.
First, suppose that $d\phi^*(\infty)\in P$. If $x$ is a point of $C_K$, then in $\widehat{\Pic(C_K)}$ we have $\sigma(x)+x=\deg(x)\phi^*(\infty).$
Thus for $D\in P$, we have $\sigma(D)=\deg(D)(\infty_1 +\infty_2)-D\in P$.  Thus $L/Q$ is Galois.  For any $D\in \Pic(C_K)$, we have 
$\sigma(D)+D=\deg(D)(\infty_1 +\infty_2)\in P$, so $\sigma$ acts as inversion on $\Gal(L/K).$

Second, suppose that $L/Q$ is Galois, and $\sigma$ acts as inversion on $\Gal(L/K).$
Let $D$ be a divisor of degree $d$.
Since $d\phi^*(\infty)=D+\sigma(D),$ it must be in the kernel of the map to  $\Gal(L/K),$ i.e. $P$.
\end{proof}
 
 Let $Q$ be a global field with a  place ${\v}$.  Let $H$ be a finite group, and $c$ a conjugacy class of $H$. We fix a separable closure $\bar{Q}_{\v}$ of the completion $Q_{\v}$.  Then, inside $\bar{Q}_{\v}$ we have the separable closure $\bar{Q}$ of $Q$.  This gives a map $\Gal(\bar{Q}_{\v}/Q_{\v})  \ra \Gal(\bar{Q}/Q)$, and in particular  distinguished decomposition and  inertia groups  in $\Gal(\bar{Q}/Q)$ at ${\v}$.  
We define (as in \cite[Section 10.2]{Ellenberg2012}) a \emph{marked $(H,c)$ extension} of $Q$ to be $(L,\pi,m)$
such that 
$L/Q$ is a Galois extension of fields, 
$\pi$ is an isomorphism $\pi: \Gal(L/Q)\isom H$ such that all inertia groups in $\Gal(L/Q)$ (except for possibly the one at ${\v}$) have image in $\{1\}\cup c$, and $m$, the \emph{marking}, is a homomorphism $L_{\v}:=L\tensor_Q Q_{\v} \ra \bar{Q}_{\v}$. 
 Note that restriction to $L$ gives a bijection between homomorphisms 
$L_{\v} \ra \bar{Q}_{\v}$ and homomorphisms $L\ra \bar{Q}$.  
Two marked $(H,c)$ extensions $(L_1,\pi_1,m_1)$ and $(L_2,\pi_2,m_2)$ are isomorphic when there is an isomorphism $L_1\ra L_2$ taking $\pi_1$ to $\pi_2$ and $m_1$ to $m_2$.  
The marking $m$ in a marked $(G',c)$ extension $(L,\pi,m)$ gives a map $\Gal(\bar{Q}_{\v}/Q_{\v}) \ra \Gal(L/Q)$.
 Composing with $\pi$ we get an \emph{infinity type} $\Gal(\bar{Q}_{\v}/Q_{\v}) \ra G'$.

\begin{proposition}\label{P:countsur}
Let $A$ be a finite abelian group of odd order, $H:=A \rtimes_{-1} S_2$ with the action of $S_2$ by inversion, and $c$ the conjugacy class of order $2$ elements of $H$.
When $K/Q$ is a quadratic extension, we have a $|A|$-to-$1$ map
$$
\{\textrm{isom. classes of marked $(H,c)$-extns $L/Q$}|  L^A\isom K \} \ra \Sur({\Pic^0(C_K)},A).
$$
When $K\in SCQ$ and $a\in A$, 
the surjections that send $\delta_K \ra  a$ correspond to those $L$ for which $2\Frob_{\infty_1}\mapsto  a \in \Gal(L/Q)$.  In particular, 
we have for $a\in A$:
$$
|\Sur(({\Pic^0(C_K)},\delta_K),(A,a))|=\frac{\#\{\textrm{isom. classes of marked $(H,c)$-extns $L/Q$}|  2\Frob_{\infty_1}\mapsto  a, L^A\isom K \}}{|A|}.
$$
Moreover, for the $L$ above we have $\Nm \Disc(L)=\Nm \Disc(K)^{|A|}$.
\end{proposition}

\begin{proof}

We have a natural identification of 
$\Sur({\Pic^0(C_K)},A)$ and $$\Sur({\Pic(C_K)/\langle \phi^*(\infty) \rangle},A),$$ since we have
the exact sequence
$$
1\ra \Pic^0(C_K) \ra \Pic(C_K)/\langle \phi^*(\infty) \rangle \ra \Z/2\Z.
$$

 We will now construct the map.
 Note that in each isomorphism class of  marked $(H,c)$ extensions of $Q$, there is a distinguished element such that $L\sub \bar{Q}$ and $m|_{L}$ is the inclusion map.  
We start with a distinguished
$$
L \in \{\textrm{isom. classes of marked $(H,c)$-extns $L/Q$}| L^A\isom K,  2\Frob_{\infty}\mapsto \pm a \}.
$$
The $H$ structure gives an isomorphism $\Gal(L/K)\isom A$. 
Since $L/Q$ is an $(H,c)$-extension, this implies that $L/K$ is abelian and unramified, so we have $L\sub K^{un,ab}$.  
This gives us a surjection $\widehat{\Pic(C_K)}\ra A$.  We can see that the surjection sends $\phi^*(\infty)$ to $0$ by Proposition~\ref{P:whichL}, and also when $K\in SCQ$, we can trace the image of $\Frob_\infty$ by the argument above. 

We have that $\Gal(K^{un,ab}/K)\isom \widehat{\Pic(C_K)}.$   From Proposition~\ref{P:whichL}, we see that a surjection 
$\Pic(C_K)/\langle \phi^*(\infty)\rangle \ra A$ corresponds exactly to an $L$ with $K\sub L\sub K^{un,ab}$  and an isomorphism
$\Gal(L/K)\isom A$ such that $L/Q$ is Galois and $\Gal(K/Q)=\langle \sigma\rangle $ acts as inversion on $\Gal(L/K)$.
In particular when $K\in SCQ$, in $\Gal(L/Q)$, we have that $\delta_K=\Frob_{\infty_1}-\Frob_{\infty_2}=\Frob_{\infty_1}-\sigma(\Frob_{\infty_1})=2\Frob_{\infty_1}.$

Since $L/Q$ is Galois and
 $\sigma$  acts as $-1$ on $\Gal(L/K)$, we have $\Gal(L/Q)\isom A \rtimes_{-1} S_2$.
Moreover, there are $|A|$ choices of an isomorphism $\Gal(L/Q)\isom A \rtimes_{-1} S_2$
that restrict to the given choice of $\Gal(L/K)\isom A$, each determined by which element of $H\setminus A$ goes to $(1,\tau)$,
where $\tau$ is the generator of $S_2$. 
 The fact that $L/K$ is unramified implies that we obtain, for each of these $|A|$ choices, a marked $(H,c)$-extension, where the marking is by the inclusion
 into $\bar{Q}$.  


Since $L/K$ is unramified, we have the statement on their discriminants.

We can check that the constructions given above are inverse to each other, which completes the proof of the proposition.
\end{proof}

\subsection{Group theory definitions}\label{S:group}
In this section, we will give necessary group theorem definitions for using the results from \cite{Ellenberg2012}.
In this section, we will work with a finite group $H$. 

Given a finite group $H$ and a conjugacy class $c$  of $H$, we 
 will define the \emph{universal marked central extension} $\widetilde{H}$ of $H$ (with respect to $c$), following
\cite[Section 7]{Ellenberg2012}.  \emph{In this section, we suppose that if $[g]\in c$ and $d$
is relatively prime to the order of $g$, then $[g^d]\in c$. }(If this is not the case, more complicated definitions are required.) Let $\SC$ be a Schur cover of $H$ so we have an exact sequence
$$
1\ra H_2(H,\Z)\ra \SC \ra H \ra 1
$$
by the Schur covering map.
For $x,y\in H$ that commute, let $\hat{x}$ and $\hat{y}$ be arbitrary lifts to $\SC$, and let $\langle x,y\rangle$
be the commutator $[\hat{x},\hat{y}]\in \SC$, which actually lies in 
$ H_2(H,\Z)$ since $x$ and $y$ commute.  
  It we take the quotient of the above exact sequence by all $\langle x,y\rangle$ for $x\in c$ and $y$ commuting with $x$, we obtain an exact sequence 
\begin{equation}\label{E:defH}
1\ra H_2(H,c)\ra \widetilde{H}_c \ra H \ra 1,
\end{equation}
which is still a central extension, defining $H_2(H,c)$ and  $\widetilde{H}_c$ . 
Let $(H)^{ab}$ denote the abelianization of $H$.
 The universal marked central extension is $\widetilde{H}:=\widetilde{H}_c\times_{(H)^{ab}} \Z$ and the map 
 $\Z \ra (H)^{ab}$ sends
$1$ to an image of an element of $c$.  We have a map $\widetilde{H} \ra H$, given through projecting to the first factor.  (See \cite[Section 7]{Ellenberg2012} for why this is called a universal marked central extension.)

 Let $\hat{\Z}$ be the inverse limit 
$\varprojlim \Z/n\Z$ taken over $n$ relatively prime to $q$ (we follow the notation of \cite{Ellenberg2012} instead of the more customary $\hat{\Z}'$). 
We are now going to define an action of $\hat{\Z}^\times$ on $\widetilde{H}$, called the \emph{discrete action} \cite[Section 8.1.7, Equation 9.4.1]{Ellenberg2012}. 
There is an action of $\hat{\Z}^\times$ on $H$ given by powering.    
We pick one element $g\in c$ and one lift $\hat{g}\in \widetilde{H}_c$ of $g$.  Next we will extend this to a map $\hat{}: c \ra \widetilde{H}_c$ such that for all $g\in c$, we have $\hat{g}$ has image $g$ in $H$.  
We define $\widehat{h g h^{-1}}=\tilde{h} \hat{g} \tilde{h}^{-1}$ for any choice of lift $\tilde{h}\in\widetilde{H}_c $ of $h$,
and since $\widetilde{H}_c\ra H$ is central, the definition does not depend on our choice of lift.
For $\alpha\in\hat{\Z}^\times$
$$
z(\alpha)=\hat{g}^{-\alpha} \widehat{g^{\alpha}}.
$$
First, we note that $z(\alpha)$ is defined by a product in $\widetilde{H}_c$, but actually lies in 
$H_2(H,c)$ since its image in $H$ is trivial.  Second, one can work out that $z(\alpha)$ does not depend on the choice of $g\in c$ (see \cite[Section 3.1]{Wood2017a}).  

   The discrete action of $\hat{\Z}^\times$ on $\widetilde{H}$  is given by
$$
\alpha * (g,m  )=(g^{\alpha}z(\alpha)^m,m).
$$

\subsection{Properties of the Hurwitz scheme}

In this theorem, we recall the properties of the Hurwitz scheme constructed by Ellenberg, Venkatesh, and Westerland, building on work on Romagny and Wewers \cite{Romagny2006}, as well as results on its homological stability from \cite{Ellenberg2016} and components \cite{Ellenberg2012}.  An extension $L/\F_q(t)$ is \emph{regular} if it does not contain a non-trivial base field extension $\F_{q^r}(t)/\F_q(t).$

\begin{theorem}[Ellenberg, Venkatesh, and Westerland]\label{T:Chur}\footnote{The paper \cite{Ellenberg2012} has been temporarily withdrawn by the authors because of a gap which affects
Sections 6, 12 and some theorems of the introduction of \cite{Ellenberg2012}. That gap does not affect any of the
results from \cite{Ellenberg2012} that we use in this paper.}
Let $H$ be a finite group with trivial center and let $c$ be a conjugacy class of order $2$ elements of $H$,  such that the elements of $c$ generate $H$. 
 Let $\F_q$ be a finite field with $q$ relatively prime to $| H|$.
There is a Hurwitz scheme $\CCHur_{H,n}$ over $\Z[| H|^{-1}]$ constructed in \cite[Section 8.6.2]{Ellenberg2012} with the following properties:
 \begin{enumerate}

\item We have $\CCHur_{ H,n}$ is a finite \'etale cover of the relatively smooth $n$-dimensional configuration space $\Conf^n$ of $n$ distinct unlabeled points in $\mathbb{A}^1$ over $\Spec \Z[| H|^{-1}]$.
\label{i:smooth} 
 
\item There is an action of $H$ on $\CCHur_{ H,n}$. 
 
 \item  \label{i:bij} 
 The scheme $\CCHur_{ H,n}$ has an open and closed subscheme $\CCHur^{c,1}_{ H,n}$ such that for $h\in H$ there is a bijection between
 \begin{enumerate}
 \item isomorphism classes of  regular marked $(H,c)$-extensions $M$ of $\F_q(t)$ 
 with unramified infinity type $\phi$ such that $\phi(F_\Delta)=h$ and
 such that the total degree of ramified non-infinite places of $\F_q(t)$ is $n$
(where $F_\Delta$ is a lift of Frobenius to $\Gal(\overline{\F_q(t)_{\v}}/\F_q(t)_{\v})$ that acts trivially on $\F_q((t^{-1/\infty}))$). 
 \item points of $s\in\CCHur^{c,1}_{ H,n}(\bar{\F_q})$ such that $h^{-1}\Frob(s)=s$ 
  \cite[Section 10.4]{Ellenberg2012}.
 \end{enumerate}
 
 \item We have $\CCHur_{H,n}(\C)$ is homotopy equivalent to a topological space $\CHur_{H,n}$ \cite[Section 8.6.2]{Ellenberg2012}, 
such that 
for any field $k$ of characteristic relatively prime to $|\GG|$,
there is a constant $C$ such that for all $i\geq 1$ and for all $n$ we have $\dim H^i(\CHur_{H,n},k)\leq C^i$ \cite[Proposition 2.5 and Theorem 6.1]{Ellenberg2016}. \label{i:topbounds}

\item Given $H$, for all $n$ sufficiently large and all $q$ with $(q,|H|)=1$, for $h\in H$ the $h^{-1}\Frob$  fixed components of $\CCHur^{c,1}_{H,n}  \tensor_{\Z[|\GG|^{-1}]} {\bar{\F}_q}$ are in bijection with elements
$(x,n) \in \widetilde{H} $ such that $q^{-1}*(x,n)=\hat{h}^{-1} (x,n) \hat{h}$ (where $\hat{h}$ is any lift of $h$ to $\widetilde{H}$) and $x$ has trivial image in $H$  \cite[Theorem 8.7.3]{Ellenberg2012} (see Section~\ref{S:group} for definitions). 
\label{i:comp}

 \end{enumerate}
\end{theorem} 

\begin{remark}
The scheme $\CCHur^{c,1}_{H,n}$ comes from restricting to the parametrization of covers of $\P^1$  all of whose local inertia groups have image 
 in $c \cup \{1\}$ and that are unramified at $\infty$.  The argument that $\CCHur^{c,1}_{H,n}$ is an open and closed subscheme is as in \cite[Section 7.3]{Ellenberg2016}. 
Our description of the components requires a bit of translation from that in \cite[Theorem 8.7.3]{Ellenberg2012}.  They biject the components with $\hat{\Z}^\times$ equivariant functions from topological generators of 
$\varprojlim \mu_n$ (taken over $n$ relatively prime to $q$) to the preimage of $1$ in $\widetilde{H}$ that are fixed by the action of $h^{-1} \Frob$ on $\varprojlim \mu_n$.
  By choosing any topological generator of $\varprojlim \mu_n$, its image under a function to $\widetilde{H}$ gives us a corresponding element of $\widetilde{H}$, and the action of $\Frob$ corresponds to the above discrete action of $q^{-1}$ on $\widetilde{H}$.
\end{remark}

\subsection{Proof of Theorem~\ref{T:RSC}} 
We continue the notation from Section~\ref{S:FFnot}.
Let $H:=A\rtimes_{-1} S_2$ and $Q=\F_q(t)$.
In Proposition~\ref{P:countsur}, given $K$, we have
\begin{align*}
&\#\{\textrm{isom. classes of marked $(H,c)$-extns $L/Q$}|  2\Frob_{\infty_1}\mapsto  a, L^A\isom K \}\\
=&\#\{\textrm{isom. classes of marked $(H,c)$-extns $L/Q$}|  2\Frob_{\infty_2}\mapsto  a, L^A\isom K \}.
\end{align*}
The $F_\Delta$ in Theorem~\ref{T:Chur} is a Frobenius element in $\Gal(L/Q)$ and since $K/Q$ is split completely, we 
have $F_\Delta=\Frob_{\infty_1}$ or $\Frob_{\infty_2}$. 
So we can sum Proposition~\ref{P:countsur} over $K\in SCQ$ to obtain
\begin{align*}
&\sum_{K\in SCQ, \Nm \Disc(K/Q)=q^{2m}} |\Sur(({\Pic^0(C_K)},\delta_K),(A,a))|\\
=&\frac{\#\{\textrm{isom. classes of marked $(H,c)$-extns $L/Q$}|  L^A \in SCQ , 2F_\Delta\mapsto  a\}}{|A|}.
\end{align*}

We will now see that all of the $L$ that appear above are regular.  If $L/Q$ contains an extension of $\F_q(t)$, then it corresponds to some cyclic quotient of $H$.  However, the abelianization of $H$ is $S_2$, so this can only happen when $L^A$ is $\F_{q^2}(t)$. However,
$\F_{q^2}(t)$ is not split completely over $\infty$. 

By Theorem~\ref{T:Chur} \eqref{i:bij} with $h=a/2$ and $n=2m$, we then have
\begin{align*}
\sum_{K\in SCQ, \Nm \Disc(K/Q)=q^{2m}} |\Sur(({\Pic^0(C_K)},\delta_K),(A,a))|
=\frac{\#\{
s\in\CCHur^{c,1}_{ H,2m}(\bar{\F}_q)\,|\,h^{-1}\Frob(s)=s
\}}{|A|}.
\end{align*}
Let $X:= \CCHur^{c,1}_{ H,2m} \otimes_{\Z[|H|^{-1}]} \F_q$.  If $h=1$, we would need to then count $\F_q$ points of $X$.  However, for general $h$, we need to count $\F_q$ points of a different variety $Y$ over $\F_q$ such that $Y\tensor_{\F_q} \bar{\F}_q\isom 
X\tensor_{\F_q} \bar{\F}_q$.  We can descend $X':=X\tensor_{\F_q} \bar{\F}_q$ to $Y$ over $\F_q$ using the action of
$\Gal(\bar{\F}_q/\F_q)$ on $X'$ in which the Frobenius in the Galois group acts as $h^{-1}\Frob$ on $X'$, where $\Frob$ is the action of Frobenius using $X$ as the $\F_q$ structure of $X'$ (see, e.g. \cite[Corollary 16.25]{Milne}). So the Frobenius action on $Y(\bar{\F}_q)$ corresponds to the  action of $h^{-1}\Frob$ on $X(\bar{\F}_q).$
Thus, we have
\begin{align*}
\sum_{K\in SCQ, \Nm Disc(K/Q)=q^{2m}} |\Sur(({\Pic^0(C_K)},\delta_K),(A,a))|
=\frac{\#\
Y (\F_q)}{|A|}.
\end{align*}

We will apply the Grothendieck-Lefschetz trace formula to $X'\isom Y\tesnor_{\F_q} \bar{\F}_q$.
 By Theorem~\ref{T:Chur} \eqref{i:smooth}, we have that $X'$ is smooth of dimension $2m$.
 We have that $\dim H^i_{\textrm{c,\'et}}(X',\Q_\ell)=
 \dim H^{4m-i}_{\textrm{\'et}}(X',\Q_\ell)$ by Poincar\'{e} Duality.
 
Next, we will relate $\dim H^{j}_{\textrm{\'et}}(X',\Q_\ell)$ to 
 $\dim H^{j}( \CCHur^{c,1}_{H,2m}(\C),\Q_\ell)$ for some $\ell>2m$.  To compare \'{e}tale cohomology in characteristic $0$ and positive characteristic, we will use \cite[Proposition 7.7]{Ellenberg2016}.  The result \cite[Proposition 7.7]{Ellenberg2016} gives an isomorphism of \'{e}tale cohomology between characteristic $0$ and positive characteristic in the case of a finite cover of a complement of a reduced normal crossing divisor in a smooth proper scheme.  Though \cite[Proposition 7.7]{Ellenberg2016} is only stated	 for \'{e}tale cohomology with coefficients in $\Z/\ell\Z$, the argument goes through identically for coefficients in $\Z/\ell^k\Z$, and then we can take the inverse limit and tensor with $\Q_{\ell}$ to obtain the result of \cite[Proposition 7.7]{Ellenberg2016} with $\Z/\ell\Z$ coefficients replaced by $\Q_\ell$ coefficients.  So we apply this strengthened version to conclude that 
$ \dim H^{j}_{\textrm{\'et}}(X',\Q_\ell)=\dim H^{j}_{\textrm{\'et}}(( \CCHur^{c,1}_{H,2m})_{\C}),\Q_\ell)$.
(As in \cite[Proof of Proposition 7.8]{Ellenberg2016}, we apply comparison to $ \CCHur^{c,1}_{H,2m} \times_{\Conf^{2m} } \operatorname{PConf}_{2m},$ where $\operatorname{PConf}_{2m}$ is the moduli space of $2m$ labelled points on $\mathbb{A}^1$, and is the complement of a relative normal crossings divisor in a smooth proper scheme \cite[Lemma 7.6]{Ellenberg2016}.
Then we take $S_{2m}$ invariants to compare the \'etale cohomology of $\CCHur^{c,1}_{H,2m}$ across characteristics.)
By the comparison of \'{e}tale and analytic cohomology \cite[Expos\'e XI, Theorem 4.4]{1973} $\dim H^{j}( \CCHur^{c,1}_{H,2m}(\C),\Q_\ell)=\dim H^{j}_{\textrm{\'et}}(( \CCHur^{c,1}_{H,2m})_{\C}),\Q_\ell)$.

By Theorem~\ref{T:Chur}~\eqref{i:topbounds}, there is a constant $C$ such that for all $j\geq 1$ and for all $m$, we have $\dim H^{j}(\CCHur^{c,1}_{H,2m}(\C),\Q_\ell)\leq C^j$.  Thus $ \dim H^{j}_{\textrm{\'et}}(X',\Q_\ell)\leq C^j$ for all $j\geq 1$.
  Thus using Poincar\'e duality, 
$ \dim H^{i}_{c,\textrm{\'et}}(X',\Q_\ell)\leq C^{4m-i}$ for all $i<4m$.

We will apply Grothendieck-Lefschetz for the Frobenius map from $Y$ on $X'$, which is $h^{-1}\Frob$ (where $\Frob$ is the Frobenius map from $X$).
By Theorem~\ref{T:Chur} \eqref{i:comp}, we have that the number of components of $X'$ fixed by $h^{-1}\Frob$ for  $2m\geq n_H$ for some fixed $n_H$ is equal to the number of $(x,2m)\in \tilde{H}$ with $x\in H_2(H,c)$ and
$q^{-1}*(x,2m)=\hat{h}^{-1}(x,2m)\hat{h}$.  Since $x$ is central in $\tilde{H}_c$, this is the same as the number of 
$(x,2m)\in \tilde{H}$ with $x\in H_2(H,c)$ and
$q^{-1}* (x,2m)=(x,2m),$ which is $\#H_2(H,c)[q-1]$ by \cite[Proposition 3.1]{Wood2017a}. 
By \cite[Example 9.3.2]{Ellenberg2012}, we have that $\#H_2(H,c)$ is a quotient of $H_2(A,\Z)$ so if $(q-1,|A|)=1$ then 
$\#H_2(H,c)[q-1]=1.$ 

 So we have
$$
\#Y({\F}_q) =\sum_{j\geq 0} (-1)^j\Tr(h^{-1}\Frob|_{ H^j_{\textrm{c,\'et}}(X',\Q_\ell)})
$$
and 
also we know $\Tr(h^{-1}\Frob|_{ H^{4m}_{\textrm{c,\'et}}((X',\Q_\ell)})$ is $q^{2m}$ times the number of components of $X'$ fixed by $h^{-1}\Frob$.
Since $X$ is smooth, we have that the absolute value of any eigenvalue of $\Frob$ on $H^j_{\textrm{c,\'et}}(X',\Q_\ell)$ is at most
$q^{j/2}$ in absolute value, and since $h^{-1}$ is finite order and commutes with $\Frob$ 
 the same is true for eigenvalues of $h^{-1}\Frob$.
Thus, for  $2m\geq n_A$,
\begin{align*}
|\#Y({\F}_q)  - q^{2m}|&=\left| \sum_{0\leq j< 2\dim X} (-1)^j\Tr(\Frob|_{ H^j_{\textrm{c,\'et}}(X_{\bar{\F}_q},\Q_\ell)}) \right| \\
&\leq \sum_{0\leq j< 2\dim X} q^{j/2} C^{4m -j}\\
&\leq q^{2m} \sum_{1\leq i } (\sqrt{q}/C)^{-i} .
\end{align*}

Thus, for fixed $q>C^2$,
\begin{align*}
&\limsup_{m\ra\infty} \frac{ \sum_{K\in SCQ, \Nm \Disc(K/Q)=q^{2m}} |\Sur(({\Pic^0(C_K)},\delta_K),(A,a))|}{\sum_{K\in SCQ, \Nm \Disc(K/Q)=q^{2m}} 1}\\
=
&\limsup_{m\ra\infty} \frac{ \sum_{K\in SCQ, \Nm \Disc(K/Q)=q^{2m}} |\Sur(({\Pic^0(C_K)},\delta_K),(A,a))|}{q^{2m}-q^{2m-1}}
\\
=
&\limsup_{m\ra\infty} \frac{q^{2m}+O(\frac{q^{2m}}{\sqrt{q}/C-1}  )  }{(q^{2m}-q^{2m-1})|A|}\\
=
&\limsup_{m\ra\infty} \frac{1+O(\frac{1}{\sqrt{q}/C-1}  )  }{(1-q^{-1})|A|}.
\end{align*}
The implied constant in the big $O$ notation is $1$.
A similar argument works for the $\liminf$ and the theorem follows.

\subsection{Corollaries}

Note that by adding Theorem~\ref{T:RSC} over all elements $a\in A$, we have the following corollary, which gives  evidence that among real quadratic $K/\F_q(t)$, the Sylow $p$-subgroups $\Pic^0(C_K)_p$ are distributed according the measure $\mu^0$ defined in Section~\ref{S:meas}.
\begin{corollary}\label{C:Picmom}
 For a finite odd order abelian group $A$, let
 $$
\delta^+_q:=\limsup_{m \ra\infty}  \frac{  \sum_{K/\F_q(t) \textrm{ real quad,} \Nm \Disc(K)=q^{2m}} |\Sur(\Pic^0(C_K),A)|}{\sum_{K/\F_q(t) \textrm{ real quad,} \Nm \Disc(K)=q^{2m}}1}
$$
and $\delta^-_q$ the corresponding $\liminf$.
Then as $q\ra\infty$ among odd prime powers such that $(q(q-1),|A|)=1$, we have
$$
\delta^+_q,\delta^-_q\ra 1.
$$
\end{corollary}

By restricting Theorem~\ref{T:RSC} to the case $a=0$ and using the fact that
$\Cl(\O_K)\isom\Pic^0(C_K)/\langle\infty_1-\infty_2 \rangle,$ we obtain the following corollary, which gives  evidence that among real quadratic $K/\F_q(t)$, the groups $\Cl(\O_K)_p$ are distributed according the measure $\mu^1$ defined in Section~\ref{S:meas}.
\begin{corollary}\label{C:regmom}
 For a finite odd order abelian group $A$, let
 $$
\delta^+_q:=\limsup_{m \ra\infty}  \frac{  \sum_{K/\F_q(t) \textrm{ real quad,} \Nm \Disc(K)=q^{2m}} |\Sur(\Cl(\O_K),A)|}{\sum_{K/\F_q(t) \textrm{ real quad,} \Nm \Disc(K)=q^{2m}}1}
$$
and $\delta^-_q$ the corresponding $\liminf$.
Then as $q\ra\infty$ among odd prime powers such that $(q(q-1),|A|)=1$, we have
$$
\delta^+_q,\delta^-_q\ra \frac{1}{|A|}.
$$
\end{corollary}

Note that it is not clear  whether the groups $\Cl(\O_K)_p$ (or $\Pic^0(C_K)_p$) are ``distributed according to a measure,'' i.e. whether there is a measure $\nu$ such that for all non-negative functions $f$, we have 
$$
\lim_{m \ra\infty}  \frac{  \sum_{K\in SCQ, \Nm \Disc(K/Q)=q^{2m}} f(\Cl(\O_K)_p)}{\sum_{K\in SCQ, \Nm \Disc(K/Q)=q^{2m}} 1} =\sum_{H \textrm{ fin ab $p$-group}} f(H)\nu(H).
$$
So even if we knew $\delta_q^\pm=1$, 
we could not use Proposition~\ref{P:momdet} to conclude that the averages of an arbitrary $f$ are predicted as by the Cohen-Lenstra heuristics.   Informally, the moments determine the measure, but we don't know whether the class groups are distributed according to a measure!
(Though for $f$ an indicator function, see  \cite[Corollary 8.2]{Ellenberg2016}.)

\section{Inert quadratic function fields}\label{S:iproof}

We continue with the notation given in Section~\ref{S:FFnot}.  
Let $INQ$ be the set
of quadratic extensions of $Q$ inert at $\infty$.  
We then have the following theorem giving  evidence that over $K\in INQ$,
the Sylow $p$-subgroups $\Pic^0(C_K)_p$ are distributed according to $\mu^0$.  
\begin{theorem}\label{T:in}

 For a finite odd order abelian group $A$, let
 $$
\delta^+_q:=\limsup_{m \ra\infty}  \frac{  \sum_{K\in INQ \Nm \Disc(K)=q^{2m}} |\Sur(\Pic^0(C_K),A)|}{\sum_{K\in INQ \Nm \Disc(K)=q^{2m}}1}
$$
and $\delta^-_q$ the corresponding $\liminf$.
Then as $q\ra\infty$ among odd prime powers such that $(q(q-1),|A|)=1$, we have
$$
\delta^+_q,\delta^-_q\ra 1.
$$

\end{theorem}

The proof below is analogous to, but easier than, the proof of Theorem~\ref{T:RSC}.

\begin{proof}
Let $H:=A\rtimes_{-1} S_2$ and $Q=\F_q(t)$.
We sum Proposition~\ref{P:countsur} over $K\in INQ$ to obtain
\begin{align*}
&\sum_{K\in INQ, \Nm \Disc(K/Q)=q^{2m}} |\Sur(\Pic^0(C_K),A)|\\
=&\#\{\textrm{isom. classes of marked $(H,c)$-extns $L/Q$}|   L^A \in INQ \}.
\end{align*}
Note that $L^A\in INQ$ if and only if in the marked extension $\pi(F_\Delta)\in H\setminus A$ and $\pi$ is unramified at $\infty$.

Almost all of the $L$ that appear above are regular.  If $L/Q$ contains an extension of $\F_q(t)$, then it corresponds to some cyclic quotient of $H$.  However, the abelianization of $H$ is $S_2$, so this can only happen when $L^A=K$ is $\F_{q^2}(t)$. Since $\Pic^0(C_{\F_{q^2}(t)})$ is trivial, it does not contribute to the sum unless $A$ is trivial (in which case the theorem is immediate).

Thus, when $A$ is non-trivial, by Theorem~\ref{T:Chur} \eqref{i:bij}, we then have
\begin{align*}
\sum_{K\in INQ, \Nm \Disc(K/Q)=q^{2m}} |\Sur(\Pic^0(C_K),A)|
=\sum_{h\in H\setminus A}
\frac{\#\{
s\in\CCHur^{c,1}_{ H,n}(\bar{\F}_q)\,|\,h^{-1}\Frob(s)=s
\}}{|A|}.
\end{align*}
We let $n=2m$.
Let $X':=X\tensor_{\F_q}\bar{\F}_q$.
As in the proof of Theorem~\ref{T:RSC}, we construct $Y_h$ so that 
$$
\#\{
s\in\CCHur^{c,1}_{ H,n}(\bar{\F}_q)\,|\,h^{-1}\Frob(s)=s
\}=\#Y_h(\F_q)
$$
and $Y\tensor_{\F_q}\bar{\F}_q\isom X'.$
As in the proof of Theorem~\ref{T:RSC},  for some prime $\ell$ we have
$ \dim H^{i}_{c,\textrm{\'et}}(X',\Q_\ell)\leq C^{2n-i}$ for all $i<2n$.

We will apply Grothendieck-Lefschetz for the Frobenius map from $Y$ on $X'$, which is $h^{-1}\Frob$ (where $\Frob$ is the Frobenius map from $X$).
As in the proof of Theorem~\ref{T:RSC}, we apply  Theorem~\ref{T:Chur} \eqref{i:comp} to conclude that the number of components of $X'$ fixed by $h^{-1}\Frob$ for even $n\geq n_H$ for some fixed $n_H$ is $1$.

 So we have, as in the proof of Theorem~\ref{T:RSC}, for even $n\geq n_A$,
\begin{align*}
|\#Y_h({\F}_q)  - q^{n}| \leq q^{n} \sum_{1\leq i } (\sqrt{q}/C)^{-i} .
\end{align*}

Thus, for fixed $q>C^2$,
\begin{align*}
&\limsup_{m\ra\infty} \frac{ \sum_{K\in INQ, \Nm \Disc(K/Q)=q^{2m}} |\Sur(\Pic^0(C_K),A)|}
{\sum_{K\in INQ, \Nm \Disc(K/Q)=q^{2m}} 1}\\
=
&\limsup_{m\ra\infty} \sum_{h\in H\setminus A} \frac{q^{2m}+O(\frac{q^{2m}}{\sqrt{q}/C-1}  )  }{(q^{2m}-q^{2m-1})|A|}\\
=
&\limsup_{m\ra\infty} \frac{1+O(\frac{1}{\sqrt{q}/C-1}  )  }{(1-q^{-1})}.
\end{align*}
The implied constant in the big $O$ notation is $1$.
A similar argument works for the $\liminf$ and the theorem follows.

\end{proof}

\section{Number field results on distribution of $\wp_1-\wp_2$}\label{S:CLd}
In this section, we see that the $(\Z/3\Z,g)$-moments predicted by Conjecture~\ref{C:atsc} hold in the number field case, and moreover are not affected by finitely many local conditions on the quadratic field, in the spirit of Conjecture~\ref{C:indoffin}.
We reduce the problem to counting cubic and quadratic extensions with certain local conditions, and then use the work of Davenport and Heilbronn \cite{DH71} to count cubic extensions.   This strategy and the computation of local masses follows along similar lines to the proof of \cite[Corollary 4]{Bhargava2016}. 
For a number field $K$, let $\O_K$ denote its ring of integers.

\begin{theorem}[Distribution of elements in $\Cl_K$, $\Z/3\Z$ moment]\label{T:DH2}
 Let $v_1,\dots,v_n$ be finite places of $\Q$, and $F_i$ be \'{e}tale quadratic $\Q_{v_i}$-algebras,
with $F_1=\Q_{v_i}^{\oplus 2}$.  
Let $S$ be the set of quadratic extensions $K$ of $\Q$ such that $K\tensor_\Q \Q_{v_i}\isom  F_i$.
Let $S^+$ and $S^-$ denote the imaginary and real extensions in $S$, respectively.
Let $v_1$ split into $w_1$ and $w_2$ in $K$ and let $g\in \Z/3\Z$.
We have
$$
\lim_{X\ra\infty} \frac{\sum_{K\in S^+, |\Disc(K)|<X} |\Sur((\Cl(\O_K),w_1-w_2),(\Z/3\Z,g)| }{\sum_{K\in S^+, |\Disc(K)|<X} 1}
=\frac{1}{3},
$$
and
$$
\lim_{X\ra\infty} \frac{\sum_{K\in S^-, |\Disc(K)|<X}  |\Sur((\Cl(\O_K),w_1-w_2),(\Z/3\Z,g)| }{\sum_{K\in S^-, |\Disc(K)|<X} 1}
=\frac{1}{9}.
$$
\end{theorem}



We can of course add these pointed moments up over the three choices of $g\in\Z/3\Z$ and recover the theorem of Davenport and Heilbronn \cite[Theorem 3]{DH71} in the case $n=0$ or the theorem of Bhargava and Varma \cite[Corollary 4]{Bhargava2016} when there are local conditions.
 (Note that $|\Sur(\Cl(\O_K),\Z/3\Z)|+1$ is the size of the $3$-torsion of 
$\Cl(\O_K)$.)

The proof of the following  proposition is similar to that of Proposition~\ref{P:countsur}.  This number field version is easier than the function field version, because for any $K/\Q$ quadratic with $\sigma$ generating $\Gal(K/\Q)$ and $L/K$ unramified abelian we have that $L/\Q$ is Galois and
$\sigma$ acts by inversion of $\Gal(L/\Q)$, and so an analog of Proposition~\ref{P:whichL} is not required.
We say a cubic extension $L$ is \emph{nowhere overramified} if no rational prime ramifies to degree $3$.  
\begin{proposition}\label{P:countsur1}
Let $c$ the conjugacy class of order $2$ elements of $S_3$.
When $K/\Q$ is quadratic, we have a $2$-$1$ map
$$
\Sur(\Cl(\O_K),\Z/3\Z) \ra
\{
\textrm{isom. classes of nowhere overram. non-cyclic cubic $L/Q$}
|  \tilde{L}^{A_3} \isom K \},
$$
given by letting $\tilde{L}$ be the unramified extension of $K$ associated to the surjection, and $L$ a cubic subfield of $\tilde{L}$ so that $\tilde{L}$ is the Galois closure of $L$ over $\Q$.  Moreover, if $v_1$ is split into $w_1$ and $w_2$ in $K$, then under the above bijection, the image of $w_1-w_2$ is trivial in $\Z/3\Z$ if and only if $L$ is split completely over $v_1$. 
\end{proposition}

Asking that $K\in S$ in Theorem~\ref{T:DH2} translates directly into conditions on the allowable 
restrictions  on the isomorphism type of $L \tesnor_\Q \Q_{v_i}$.  Thus in order to count the objects on the right hand side in Proposition~\ref{P:countsur1}, we will count isomorphism classes of cubic number fields $L/\Q$ with restrictions on the isomorphism type of $L\tesnor_\Q \Q_{v_i}$.  To ask that $L$ is nowhere overramified and non-cyclic we need 1)that $L/\Q$ is not Galois and 2)in the associated map $\phi_3: G_\Q \ra S_3$ to $L$ no inertia group has image including a $3$-cycle.  The second requirement is a condition on the isomorphism type of $L\tesnor_\Q \Q_{v}$
at every finite place $v$.

The following theorem on counting cubic extensions with local restrictions will be essential.  
For a prime $p$ of $\Q$
let $|\cdot|_p$ be the $p$-adic absolute value so that $|p|_p=p^{-1}$, and let $|\cdot|_\infty \equiv 1$ be the trivial absolute value.

\begin{theorem}[Theorem 4.1 of \cite{DW88}, see also Theorem 1 and Section 5 of \cite{DH71}]\label{T:countcubic}
 Let $v_0=\infty, v_1, \dots v_n$ be distinct places of $\Q$ and $R_i$ (for $i=0,\dots,n$) be a set
of isomorphism classes of degree $3$ \'{e}tale  $\Q_{v_i}$-algebras.  
For a place $v$ of $\Q$ and an \'{e}tale  $\Q_v$-algebra $M_v$, let $c(M_v)=|\Aut(M_v/Q_v)|^{-1} |\Disc(M_v/Q_v) |_v$, and $c(R_{i})=\sum_{M\in R_{i} } c(M)$.  Let $c_3(v)= \sum_{M/\Q_v \textrm{ deg. $3$ \'{e}tale} } c(M)$.
Then
$$
\lim_{X\ra\infty} \frac{\#\{L/\Q \textrm{ cubic, up to isom.}\ |\ |\Disc L| <X;  L\tensor_\Q \Q_{v_i}\in R_i \ i=0,\dots,n\}}{X}=\frac{1}{3\zeta(3)}\prod_{i=0}^n \frac{c(R_i)}{c_3(v_i)}.
$$
\end{theorem}

We will need a similar theorem for quadratic extensions.
\begin{theorem}[follows from Theorem 1.1 of \cite{Woo10}, see also Lemma 6.1 of \cite{Taniguchi2013}] \label{T:countquadratic}
 Let $v_0=\infty, v_1, \dots v_n$ be distinct places of $\Q$ and $R_i$ (for $i=0,\dots,n$) be a set
of isomorphism classes of degree $2$ \'{e}tale $\Q_{v_i}$-algebras.  
For a place $v$ of $\Q$ and an \'{e}tale  $\Q_v$-algebra $M_v$, let $c(M_v)=|\Aut(M_v/Q_v)|^{-1} |\Disc(M_v/Q_v) |_v$, and $c(R_{i})=\sum_{M\in R_{i} } c(M)$.  Let $c_2(v)= \sum_{M/\Q_v \textrm{ deg. $2$ \'{e}tale} } c(M)$.
Then
$$
\lim_{X\ra\infty} \frac{\#\{K/\Q \textrm{ quad., up to isom.}\ |\ |\Disc K| <X; K\tensor_\Q \Q_{v_i}\in R_i \ i=0,\dots,n\}}{X}=\frac{1}{\zeta(2)}\prod_{i=0}^n \frac{c(R_i)}{c_2(v_i)}.
$$
\end{theorem}

\begin{proof}[Proof of Theorem~\ref{T:DH2}]

For $i=2,\dots,n$, let $R_i$ be the set of cubic \'{e}tale extensions of $\Q_{v_i}$
such that the associated $\phi_3: G_{\Q_{v_i}} \ra S_3$ has sign map $\phi_2 : G_{\Q_{v_i}} \ra S_2$
corresponding to $F_i$, 
\emph{and} such that the image of inertia under $\phi_3$ does not include a $3$-cycle.  
For $i=1$, we let $R_1=\{\Q_{v_1}^{\oplus 3} \} $, and for $i=0$ we let $R_0$ be $\{\C\oplus \R \}$ or $\{ \R^{\oplus 3}\}$ depending on whether we are in the $S^+$ or $S^-$ case, respectively.
Label the places $v\not\in\{v_0,\dots,v_n\}$ by $v_{n+1}, v_{n+2},\dots$, and for $i\geq n+1$, let $R_i$ be the set of cubic \'{e}tale extensions of $\Q_{v_i}$
such that the associated $\phi_3: G_{\Q_{v_i}} \ra S_3$ does not have a $3$-cycle in its image of inertia.
Let $F_0$ be $\{\C \}$ or $\{ \R^{\oplus 2}\}$ depending on whether we are in the $S^+$ or $S^-$ case, respectively.

 We have, from Proposition~\ref{P:countsur1},
\begin{align*}
 &\lim_{X\ra\infty} \frac{\sum_{K\in S^+, |\Disc(K)|<X} |\Sur((\Cl(\O_K),w_1-w_2),(\Z/3\Z,0)| }{\sum_{K\in S^+, |\Disc(K)|<X} 1}\\
=&\lim_{X\ra\infty} \frac{ 2\#\{\textrm{isom. classes of  non-cyclic cubic  $L/\Q$}||\Disc(L_1)|<X,   L\tensor_\Q \Q_{v_i} \in R_i, \ i\geq 0 \} }{\#\{\textrm{isom. classes of quad.  $K/\Q$}||\Disc(K)|<X,   K\tensor_\Q \Q_{v_i} \isom F_i, \ i= 0,\dots,n \}}.
\end{align*}

Let
$$
N_Y(X):=\#\{\textrm{isom. classes of cubic  $L/\Q$}||\Disc(L)|<X,   L\tensor_\Q \Q_{v_i} \in R_i, \ i= 0,\dots,Y \}.
$$
By Theorem~\ref{T:countcubic}, for finite $Y\geq n$, we have
$$
\lim_{X\ra\infty} \frac{N_Y(X)}{X} = \frac{1}{3\zeta(3)}\prod_{i=0}^Y \frac{c(R_i)}{c_3(v_i)}.
$$
We have $N_\infty(X) \leq N_Y(X)$, so
$$
\limsup_{X\ra\infty} \frac{N_\infty(X)}{X} \leq \frac{1}{3\zeta(3)}\prod_{i=0}^\infty \frac{c(R_i)}{c_3(v_i)}.
$$
Also, we have
$$
N_\infty(X) \geq N_Y(X) -\sum_{i>Y} \#\{ \textrm{isom. classes of cubic  $L/\Q$}||\Disc(L)|<X, 
v_i \textrm{ ram. deg. $3$ in $L$}.
\}
$$
By \cite[Lemma 5.1]{DW88}, there is a constant $C$ such that for all $i\geq 1$, and associated prime $v_i$, we have
$$
 \frac{\#\{ \textrm{isom. classes of cubic  $L/\Q$}||\Disc(L)|<X, 
v_i \textrm{ ram. deg. $3$ in $L$}
\}}{X} \leq \frac{C}{v_i^2}.
$$
So, 
$$
\frac{N_\infty(X)}{X} \geq \frac{N_Y(X)}{X} -\sum_{i>Y} \frac{C}{v_i^2},
$$
and thus
$$
\liminf_{X\ra \infty} \frac{N_\infty(X)}{X} \geq \frac{1}{3\zeta(3)}\prod_{i=0}^\infty \frac{c(R_i)}{c_3(v_i)}
$$
and we conclude
$$
\lim_{X\ra \infty} \frac{N_\infty(X)}{X} =\frac{1}{3\zeta(3)}\prod_{i=0}^\infty \frac{c(R_i)}{c_3(v_i)}.
$$
Since 
$$
\lim_{X\ra\infty} \frac{\#\{\textrm{isom. classes of cyclic cubic  $L/\Q$}||\Disc(L)|<X\}}{X}=0
$$
by \cite[Equation (1)]{Coh54}, if we define
$$
N'_Y(X):=\#\{\textrm{isom. classes of non-cyclic cubic  $L/\Q$}||\Disc(L)|<X,   L\tensor_\Q \Q_{v_i} \in R_i, \ i= 0,\dots,Y \},
$$
we have
$$
\lim_{X\ra \infty} \frac{N'_\infty(X)}{X} =\frac{1}{3\zeta(3)}\prod_{i=0}^\infty \frac{c(R_i)}{c_3(v_i)}.
$$

By Theorem~\ref{T:countquadratic}, we have that
$$
\lim_{X\ra\infty} \frac{\#\{\textrm{isom. classes of quad.  $K/\Q$}||\Disc(K)|<X,   K\tensor_\Q \Q_{v_i} \isom F_i, \ i= 0,\dots,n \}}{X}=\frac{1}{\zeta(2)}\prod_{i=0}^n \frac{c(F_i)}{c_2(v_i)}.
$$
Thus,
$$
\lim_{X\ra\infty} \frac{\sum_{K\in S^+, |\Disc(K)|<X} |\Sur((\Cl(\O_K),w_1-w_2),(\Z/3\Z,0)| }{\sum_{K\in S^+, |\Disc(K)|<X} 1}=
\frac{2\zeta(2)}{3\zeta(3)}\prod_{i=0}^\infty \frac{c(R_i)}{c_3(v_i)} \prod_{i=0}^n \frac{c_2(v_i)}{c(F_i)}.
$$
It remains to compute the local factors.

We have that for a finite place $v$ that $c_3(v)=1+v^{-1}+v^{-2}$ 
and $c_2(v)=1+v^{-1}$ and for $i>n$ that $c(R_i)=1+v^{-1}$.  (For tame $v$ this is a simple computation with the
absolute tame Galois group, and for wild $v$ these are the $n=2,3$ cases of   Bhargava's mass formula for local fields \cite[Theorem 1.1]{Bha07}, and follow from Serre's mass formula \cite[Th\'{e}or\`{e}me 2]{Ser78},
as well.)
Also, $c_3(\infty)=2/3$ 
and $c_2(\infty)=1$.

When $2\leq i \leq n,$ we will see that $c(R_i)=c(F_i)$.
Given $\phi: G_{\Q_v}\ra S_3$ such that the inertia group does not have image containing a $3$-cycle, since the inertia group is a normal subgroup, either $|\im(\phi)|=2$ or $\phi$ is unramified.  
If $F_i$ is unramified, then every element of $R_i$ must be unramified.
If $F_i=\Q_{v_i}^{\oplus 2}$, then $R_i=\{\Q_{v_i}^{\oplus 3}, K_{v_i} \},$ where $K_{v_i}/\Q_{v_i}$ is the unramified extension of degree $3$, and $c(F_i)=1/2$ and $c(R_i)=1/6+1/3=1/2$.  
If $F_i/\Q_{v_i}$ is an unramified field extension of degree $2$, then $R_i=\{F_i\},$ and $c(F_i)=c(R_i).$  
If $F_i$ is ramified, then since $|\im(\phi)|=2$, we have $R_i=\{F_i\oplus Q_{v_i}\}$ and $c(F_i)=c(R_i).$
Now in the case $i=1$, we have $c(R_1)=1/6$ and $c(F_1)=1/2$.
So,
$$
\lim_{X\ra\infty} \frac{\sum_{K\in S^+, |\Disc(K)|<X} |\Sur((\Cl(\O_K),w_1-w_2),(\Z/3\Z,0)| }{\sum_{K\in S^+, |\Disc(K)|<X} 1}=
\frac{\zeta(2)}{3\zeta(3)} \frac{c(R_0)}{c(F_0)}  \prod_{p}^\infty \frac{1+p^{-1}}{1+p^{-1}+p^{-2}}=\frac{c(R_0)}{3c(F_0)} .
$$
Since $c(\{ \C \})=1/2$ and $c(\{\R\oplus \C \})=1/2,$ we conclude the first case of the theorem when $g=1$.
Since $c(\{ \R^{\oplus 2} \})=1/2$ and $c(\R^{\oplus 3})=1/6,$ we conclude the second case of the theorem when $g=1$.

For non-trivial $g$, we can either  subtract the pointed moments we have just proven from the known (non-pointed) moments, or make the argument as above but with $R_1$ replaced the set of the unramified cubic field extension of $\Q_{v_1}$.
Either approach tells us the average number of surjections in which $w_1-w_2$ has non-trivial image.  We note that by the automorphism of $\Z/3\Z$ there must be an equal number of surjections with each non-trivial image, and we conclude the theorem.
\end{proof}

In fact, if finitely many
places $v_i$ are required to split completely into $v_i'$ and $v_i''$ in the quadratic extension, 
 using the argument in the proof of Theorem~\ref{T:DH2}, one can count surjections where 
$\pm (v_i'$ - $v_i'')$ have any given possible list of images in $\Z/3\Z$.  However, note that
$(\Cl(\O_K), v_1'$ - $v_1'', v_2'$ - $v_2'')$ is not a well-defined ``double-pointed group,''
and so one necessarily must keep track of the images of plus or minus certain elements if tracking more than one.
Klagsbrun \cite{Klagsbrun2017a} avoids this technical issue by studying the quotients of $\Cl(\O_K)$ by the elements of $v_i' - v_i''$ instead of moments of pointed groups.

\section{Predictions of conjectures}\label{S:conjs}
In this section, we will discuss some of the predictions of Conjectures~\ref{C:indoffin} and \ref{C:atsc}.  
In the function field analog, the question of the size of $\Pic^0$ can be rephrased more geometrically.  We consider hyperelliptic curves over $\F_q$, and ask how many $\F_q$-points there are on their Jacobians (as $\Pic^0(C_K)=\Jac(C_K)(\F_q)$). 
Since the differences of $\F_q$-points of a curve give $\F_q$-points of its Jacobian, one might first guess that curves with more points would have more points in the $p$-torsion of their Jacobians.
 Conjecture~\ref{C:indoffin} has the counter-intuitive prediction that the number of $\F_q$-points on the curve, $\#C_K(\F_q)$, does not affect the number of $\F_q$-points on the $p^k$-torsion of the Jacobian, $\#\Jac(C_K)[p^k](\F_q)$ (even though all of the points in $\Jac(C_K)(\F_q)$ are torsion points).
This is because the number of $\F_q$-points on $C_K$ is determined by the completions of $K$ at the $q+1$ degree $1$ places of $\F_q(t)$.  So even if we restrict to hyperelliptic curves with the maximum number of $\F_q$-points, $2q+2$,  Conjecture~\ref{C:indoffin} predicts that $\#\Jac(C_K)[p^k](\F_q)$ is distributed just as it is for hyperelliptic curves with no points, or for all hyperelliptic curves.  In particular, the predicted average of $\#\Jac(C_K)[p^k](\F_q)$ is $k+1$

It is worth noting that restricting to hyperelliptic curves $C_K$ with the maximal number of $\F_q$-points does indeed provably make the asymptotics of the average of $\#\Jac(C_K)(\F_q)$ larger, which can be deduced from work of Taniguchi \cite[Theorem 6.7]{Tan04}.   There are at least two important caveats, which are that 1)the average size of $\#\Jac(C_K)(\F_q)$ is infinite and 2)$\Jac(C_K)(\F_q)$ includes $2$-torsion which does not follow the heuristics discussed in this paper.

Suppose we consider imaginary quadratic extensions $K$ of $\Q$ split completely at a rational prime $\ell$ into $\ell_1, \ell_2$.  Conjecture~\ref{C:atsc} then predicts that the probability that $\ell_1-\ell_2$ is trivial in $\Cl(\O_K)_p$ is $1-p^{-1}$.
We can see this since $\sum_{A\in \A} |\Aut(G)|^{-1}|G|^{-1} \prod_{k\geq 1} (1-p^{-k})=1-p^{-1}$.
  Further,  for a random group $G$ from $\mu_{CL}$ and a uniform random element $g\in G$ we have that 
$$\Prob(G\isom A|g=1)=(1-p^{-1})^{-1}\Prob(G\isom A \textrm{ and } g=1)=|\Aut(A)|^{-1}|A|^{-1}\prod_{k\geq 2}(1-p^{-k}) . $$
Thus, if we restrict to those $K$
such that $\ell_1-\ell_2$ is trivial in $\Cl(\O_K)_p$,
 Conjecture~\ref{C:atsc} predicts
  that
$\Cl(\O_K)_p$ is then distributed according to $\mu_{CL}^r$, like the Sylow $p$-subgroups of class groups of real quadratic fields.  Said another way, if among imaginary quadratic extensions $K$ of $\Q$ split completely at a rational prime $\ell$ we consider the group $\Cl(\O_K)_p/\langle \ell_1-\ell_2 \rangle$, the distribution is predicted to be $\mu_{CL}^r$, and is predicted to not change even if we restrict to only those $K$ for which
$\ell_1-\ell_2$ is trivial in $\Cl(\O_K)_p.$

We now see somewhat of a contrast to the first paragraph of this section.  We restrict to  imaginary quadratic extensions $K$ of $\Q$ split completely at a rational prime $\ell$ such that $\ell_1-\ell_2$ is non-trivial in $\Cl(\O_K)_p$ (forcing, in particular, $\Cl(\O_K)_p$ to be non-trivial).  In this case, Conjecture~\ref{C:atsc} predicts, for example, that the $p$-torsion $\Cl(\O_K)[p]$ has average size $p+p^{-1}$ (as opposed to average size $2$ among all imaginary quadratics).

\subsection*{Acknowledgements} 
The author would like to thank Jordan Ellenberg, Akshay Venkatesh,  Nigel Boston, Silas Johnson, Manjul Bhargava, Takashi Taniguchi, and J\"{u}rgen Kl\"{u}ners  for useful conversations and comments on this paper. 
The author would also like to thank the anonymous referee for detailed comments that improved the exposition of the paper.
This work was done with the support of an American Institute of Mathematics Five-Year Fellowship,
a Packard Fellowship for Science and Engineering, a Sloan Research Fellowship, and National Science Foundation grants DMS-1147782 and DMS-1301690 and CAREER grant DMS-1652116, and a Vilas Early Career Investigator Award. 

\newcommand{\etalchar}[1]{$^{#1}$}
\def\cprime{$'$}

\end{document}